\documentclass{amsart}

\usepackage{hyperref}
\usepackage{amsmath}

\usepackage{amssymb}
\usepackage{mathrsfs}
\usepackage{amscd}
\usepackage{graphicx}
\usepackage{mathdots}
\usepackage{gensymb}
\usepackage{epstopdf}
 \usepackage{todonotes}
 \usepackage{youngtab}
 \usepackage{wrapfig}
 \usepackage{bbm}
 \usepackage{bm}
 \usepackage{verbatim}
 \usepackage[margin=1.1in]{geometry}

 \usepackage[nocompress]{cite}

\usepackage{epsfig}       
\usepackage{epic,eepic}        

\newtheorem{thm}{Theorem}[section]
\newtheorem{prop}[thm]{Proposition}

\newtheorem{cor}[thm]{Corollary}

\theoremstyle{definition}

\theoremstyle{remark}
\newtheorem{remark}[thm]{Remark}

\numberwithin{equation}{section}

\newcommand{\Z}{\mathbb{Z}}
\newcommand{\Y}{\mathbb{Y}}

\newcommand{\C}{\mathbb{C}}
\newcommand{\R}{\mathbb{R}}

\begin{document}

\title[Majorization Inequalities from Logarithmic Convexity]{Majorization Inequalities from Logarithmic Convexity}
 
 \author{Colin McSwiggen}
\address{Institute of Mathematics, Academia Sinica}
\email{csm@as.edu.tw}

\author{Siddhartha Sahi}
\address{Department of Mathematics, Rutgers University}
\email{sahi@math.rutgers.edu}

\begin{abstract}
Majorization inequalities for symmetric polynomials have interested mathematicians for centuries, from the AM--GM inequality for two variables going back at least to Baudh{\={a}}yana's \emph{\'{S}ulbas{\={u}}tra} and Euclid's \emph{Elements} in the first millenium BCE, through classical results of Newton, Muirhead and Gantmacher, to more recent extensions to Schur polynomials and zonal spherical functions. These have been established case by case, with no unified approach. Although it is known that majorization inequalities follow from symmetry and convexity in the indexing partition, the difficulty of proving convexity in specific cases has left a number of outstanding conjectures inaccessible until now.

The key insight of this paper is that log-convexity provides a more versatile tool \emph{and} a unifying principle. It implies convexity and hence majorization, and it is preserved under multiplication and weighted averaging, making it well suited to inductive arguments in a wide range of settings. Using this idea, we prove new majorization inequalities for Macdonald polynomials, Jack polynomials and Heckman--Opdam hypergeometric functions, unifying existing results and resolving several open conjectures.
\end{abstract}

\maketitle

\tableofcontents

\section{Introduction}

Any nonnegative real numbers $x_1, \hdots, x_n$ satisfy
\begin{equation} \label{eqn:amgm}
\frac{1}{n}\big(x_1 + \cdots + x_n \big) \ge \big( x_1 \cdots x_n \big)^{1/n}.
\end{equation}
This is the inequality of arithmetic and geometric means, or \emph{AM--GM inequality}, one of the most ancient and iconic estimates in mathematics.  In the case $n=2$, it has been known since at least the mid-first millenium BCE; a geometric proof is implicit in the \emph{Baudh{\={a}}yana \'{S}ulbas{\={u}}tra} \cite[B\'{S}S 2.5]{SenBag1983} (c. 800--500 BCE), which reads:
\begin{center}
\includegraphics[width=290px]{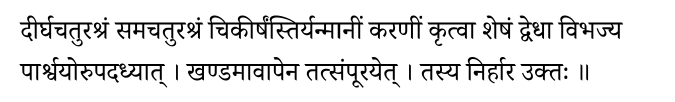}
\end{center}
This text, which forms part of a set of instructions for building altars, describes a method for constructing a square of equal area to a given rectangle. In the process, one obtains a square of equal perimeter to the original rectangle but greater area, implying that a square has maximal area among rectangles of fixed perimeter.   It is a quick and entertaining exercise to show that this is precisely the $n=2$ case of (\ref{eqn:amgm}).  Around 300 BCE, Euclid gave a similar construction with an explicitly stated inequality \cite[Propositions II.5 and VI.27]{euclid_heath_1956}, as well as a different construction that implies an inequivalent geometric proof \cite[Propositions VI.8 (Corollary) and VI.13]{euclid_heath_1956}. 

The AM--GM inequality is the oldest member of a very large family: it is the most basic example of a class of estimates known as \emph{majorization inequalities}, which have been a focus of much recent research on symmetric polynomials and special functions. This paper introduces a new technique for proving such inequalities, leading to new theorems on Macdonald polynomials, Jack polynomials and Heckman--Opdam hypergeometric functions that generalize many of the existing results in the literature and place them in a unified framework.

The study of inequalities between symmetric polynomials and related special functions often relies on integral representations that expose some underlying geometry or symmetry of the indexing parameters. Historically, many comparison results were derived using Harish-Chandra’s integral formula for spherical functions \cite{HC2}.  From the integral formula one can prove that the functions in question are convex and symmetric in the parameters, which implies that majorization inequalities must hold.  While the Harish-Chandra formula provides a powerful link between representation theory and analysis, it is inherently limited to a discrete set of cases where the polynomials or special functions have a specific group-theoretic or geometric interpretation in terms of spherical functions on a symmetric space. Consequently, this technique is fundamentally unavailable for the study of general Jack and Macdonald polynomials, which are defined over a continuous range of parameters where classical symmetric-space interpretations do not exist.

This paper offers a way to circumvent these limitations. The key insight is that the natural property to consider is not convexity in the parameters but rather the stronger yet more flexible property of \emph{log}-convexity, which turns out to hold in all cases in the literature where convexity was shown previously.  Because a product of log-convex functions is log-convex---which is not true of functions that are merely convex---this technique enables more sophisticated arguments in which symmetric functions in higher rank are constructed from those in lower rank, via formulas such as the interlacing integral representation of Macdonald polynomials developed by Okounkov and Olshanski.  Moreover, despite being a stronger property, log-convexity is sometimes easy to prove even in cases where convexity is difficult to show directly.  In addition to new results, this approach offers a systematic explanation and unified derivation for many different majorization inequalities that had previously been observed in symmetric function theory.

\subsection{Background}

Around the turn of the 20th century, Muirhead \cite{Muirhead} proved a substantial generalization of the AM--GM inequality.  Muirhead's inequality gives a comparison between pairs of monomial symmetric polynomials, which is governed by the majorization partial order on partitions.  For our purposes, a partition with $n$ parts is a vector $\lambda = (\lambda_1 \ge \cdots \ge \lambda_n \ge 0) \in \Z^n$ with weakly decreasing, nonnegative integer coordinates.  Given two partitions $\lambda, \mu \in \Z^n$, we say that $\lambda$ \emph{majorizes} $\mu$ if 
\begin{equation} \label{eqn:maj-eq}
\sum_{i=1}^n \lambda_i = \sum_{i=1}^n \mu_i
\end{equation}
and
\begin{equation} \label{eqn:maj-ineq}
\sum_{i=1}^r \lambda_i \ge \sum_{i=1}^r \mu_i, \qquad 1 \le r < n.
\end{equation}
For any partition $\lambda$, the corresponding \emph{monomial symmetric polynomial} is
\[
m_\lambda(x) = \sum_{\eta \in S_n \cdot \lambda} x^{\eta}
\]
where $S_n$ is the symmetric group on $n$ letters, $S_n \cdot \lambda$ is the $S_n$-orbit of $\lambda$, $x = (x_1, \hdots, x_n)$ is an $n$-tuple of formal variables, and $x^{\eta} = x_1^{\eta_1} \cdots x_n^{\eta_n}$.  Write $\bm{1}$ for the vector $(1, \hdots, 1) \in \R^n$. Muirhead's inequality\footnote{Muirhead's original formulation of the inequality allows $\lambda$ and $\mu$ to be real vectors rather than partitions, but this more general statement is easily deduced by approximation from (\ref{eqn:muirhead}).} states that
\begin{equation} \label{eqn:muirhead}
\frac{m_\lambda(x)}{m_\lambda(\bm{1})} \ge \frac{m_\mu(x)}{m_\mu(\bm{1})} \qquad \text{for all } x \in [0,\infty)^n \text{ if and only if $\lambda$ majorizes $\mu$.}
\end{equation}
Taking $\lambda = (n,0,\hdots,0)$, $\mu = \bm{1}$ in (\ref{eqn:muirhead}) and substituting $x_i = y_i^{1/n}$ recovers the AM--GM inequality.

Over the past several centuries, mathematicians have proved many other inequalities between symmetric polynomials.  Important examples include Newton's \cite[p.~173]{Newton1732} and Maclaurin's \cite{Maclaurin1729} inequalities for elementary symmetric polynomials, Schur's inequality for complete homogeneous symmetric polynomials \cite[p.~164]{HardyLittlewoodPolya1934}, and Gantmacher's inequality for power sums \cite[p.~203]{Gantmacher1959}. Yet even in this classical area, a major breakthrough arrived unexpectedly in 2011 with a remarkable set of results by Cuttler, Greene and Skandera \cite{CGS}.  They showed that, just like the AM--GM inequality, many other classical comparison results for symmetric polynomials are actually special cases of more general theorems that take the same form as Muirhead's inequality.  That is, for many of the classical families of symmetric polynomials that are indexed by partitions, there are pointwise inequalities between pairs of such polynomials that hold if and only if one of the partitions majorizes the other.

The majorization order thus appears as a unified organizing principle that controls pointwise inequalities between symmetric polynomials in a large number of cases.  Recent works by Sra \cite{Sra1}, Ait-Haddou and Mazure \cite{AitHaddouMazure}, McSwiggen and Novak \cite{MN-majorization}, Chen and Sahi \cite{ChenSahi2024}, Chen, Khare and Sahi \cite{ChenKhareSahi}, Xu and Yao \cite{XuYao2026}, and Khare and Tao \cite{KhareTao2021} have expanded this principle to include many further inequalities for symmetric polynomials and other special functions, and have also identified some intriguing cases in which majorization inequalities do \emph{not} hold.  Yet no satisfactory explanation has emerged for why the majorization order plays such a ubiquitous role, and proofs have proceeded largely on a case-by-case basis.

This paper offers a unifying explanation for a large class of these results: they arise from the fact that the polynomials in question are log-convex functions of the partition $\lambda$, at particular specializations $x \in \R^n$. For polynomials satisfying an appropriate symmetry condition, this log-convexity is sufficient to imply Muirhead-type inequalities.  Moreover, log-convexity is often easier to establish directly than the inequalities themselves, despite the fact that it is actually a stronger property.  Using this idea, we prove new majorization inequalities for Macdonald polynomials, Jack polynomials, and Heckman--Opdam hypergeometric functions, resolving several open conjectures and unifying a number of known results as special or limiting cases of more general statements.

\subsection{An example}

To illustrate the technique, we give a short alternate proof of a result on power sums due to Cuttler, Greene and Skandera \cite{CGS}, which generalizes the above-mentioned inequality of Gantmacher.  We first review some definitions.  Given an integer $k \ge 0$, define
\[
p_k(x) = \sum_{i=1}^n x_i^k.
\]
For any partition $\lambda$, the corresponding \emph{power sum symmetric polynomial} is then
\[
p_\lambda(x) = \prod_{i=1}^n p_{\lambda_i}(x).
\]
Let $X$ be a convex subset of $\R^n$. A function $f: X \to [0,\infty)$ is \emph{log-convex} if $\log \circ f$ is convex, with the convention $\log 0 = -\infty$.  Equivalently, for all $x, y \in X$ and all $t \in (0,1)$, we have
\[
f \big( tx + (1-t)y \big) \le f(x)^t f(y)^{1-t}.
\]
Write
\[
\mathbb{Y}_n = \big\{ \lambda \in \Z^n \ : \ \lambda_1 \ge \hdots \ge \lambda_n \ge 0 \big\}
\]
for the set of partitions with $n$ parts.  We say that a function $f : \mathbb{Y}_n \to [0,\infty)$ is log-convex if $f$ extends to a log-convex function on a convex set containing $\mathbb{Y}_n$.

The majorization order on partitions extends to a partial order on vectors as follows.  For any $x \in \R^n$, write $x^+$ for the vector of coordinates of $x$ sorted in weakly decreasing order, that is, $\{x_1^+, \hdots, x_n^+\} = \{x_1, \hdots, x_n\}$ and $x_1^+ \ge \hdots \ge x_n^+$.  For two vectors $x, y \in \R^n$, we say that $x$ \emph{majorizes} $y$ if $x^+$ and $y^+$ satisfy the same equation (\ref{eqn:maj-eq}) and inequalities (\ref{eqn:maj-ineq}) that define the majorization order on partitions, that is, if
\[
\sum_{i=1}^n x_i = \sum_{i=1}^n y_i
\]
and
\[
\sum_{i=1}^r x^+_i \ge \sum_{i=1}^r y^+_i, \qquad 1 \le r < n.
\]
There are many other equivalent definitions of the majorization order on vectors; we refer the reader to the book by Marshall, Olkin and Arnold \cite{MO}, which is the canonical reference on the theory of majorization and related inequalities.  Finally, a function $f$ from $\R^n$ or $\mathbb{Y}_n$ to $\R$ is \emph{Schur-convex} if $f(x) \ge f(y)$ whenever $x$ majorizes $y$.

To demonstrate how majorization inequalities follow from log-convexity, we now give a very brief proof of the following Muirhead-type inequality for power sums, which comprises part of \cite[Theorem 4.2]{CGS}:

\begin{thm}[Cuttler--Greene--Skandera]
    For partitions $\lambda, \mu \in \mathbb{Y}_n$, if $\lambda$ majorizes $\mu$ then $p_\lambda(x) \ge p_\mu(x)$ for all $x \in [0,\infty)^n$.
\end{thm}

The converse statement, that $p_\lambda(x) \ge p_\mu(x)$ for all $x \in [0,\infty)^n$ implies $\lambda$ majorizes $\mu$, is usually considered the ``easy direction'' for this type of result and is also shown in \cite[Theorem 4.2]{CGS}. The log-convexity technique demonstrated here applies to the ``hard direction'' of implication, that is, showing that majorization of partitions implies inequalities for the polynomials.  Although this technique yields an extremely straightforward proof in the case of power sums, for other families of polynomials the desired log-convexity and symmetry properties are less obvious, and additional ideas are required to show that they hold.

\begin{proof}
    We consider $x \in (0,\infty)^n$; the statement for $x \in [0,\infty)^n$ then follows by continuity.  The claim can be rephrased as follows: given any $x \in (0,\infty)^n$, the function $\lambda \mapsto p_\lambda(x)$ is Schur-convex.  We first observe that this function is log-convex.  Indeed, for any $1 \le i,j \le n$, the function $\lambda \mapsto x_i^{\lambda_j}$ extends to a log-linear, and thus log-convex, function on $\R^n$. Since the set of log-convex functions on a given domain is closed under products and positive linear combinations, the functions
    \[
    \lambda \mapsto p_{\lambda_j}(x) = \sum_{i=1}^n x_i^{\lambda_j} \qquad \text{and} \qquad \lambda \mapsto p_{\lambda}(x) = \prod_{j=1}^n p_{\lambda_j}(x) 
    \]
    are also log-convex.  Moreover, $p_\lambda(x)$ is manifestly symmetric in the entries of $\lambda$. Since log-convexity is stronger than convexity, and a symmetric, convex function on $\R^n$ is Schur-convex \cite[Proposition C.2, p.~97]{MO}, we have shown the claim.
\end{proof}

The crux of the proof above is that $\lambda \mapsto p_\lambda(x)$ extends to a symmetric, convex function on $\R^n$, and such a function is automatically Schur-convex.  Technically, for the purpose of showing majorization inequalities, it is therefore adequate to prove convexity in the parameter $\lambda$ rather than log-convexity (\emph{after} extending to $\R^n$, as a convex function on $\Y_n$ does not necessarily extend symmetrically to a convex function on the entire space). Indeed, the technique of using convexity and symmetry to show Schur-convexity was used previously by Sra \cite{Sra1} to prove a Muirhead-type inequality for Schur polynomials and by McSwiggen and Novak \cite{MN-majorization} to prove inequalities for spherical functions on symmetric spaces.  The above argument already illustrates two good reasons for considering log-convexity instead.  First, it is a stronger property that holds naturally at nonnegative specializations for many families of polynomials indexed by partitions.  This is especially transparent in cases like the monomial and power sum symmetric polynomials, where the entries of the partition appear as exponents of the variables $x_1, \hdots, x_n$.  Second, log-convex functions enjoy helpful properties that do not always hold for convex functions, notably closure under products. As a result, it is sometimes easy to show the stronger property of log-convexity even in cases where it would be difficult to verify convexity directly.  This is already the case for power sums, and it plays a key role in the proofs below, where the inductive arguments rely crucially on the fact that products of log-convex functions are log-convex.

\subsection{Overview}

The remainder of the paper consists of the statements and proofs of our main results.  The first of these, Theorem \ref{thm:conv-mac}, shows that when appropriately normalized and evaluated at points along certain rays in $\R^n$, Macdonald polynomials are log-convex and Schur-convex functions of the partition $\lambda$.  All other results in this paper are in some sense downstream of this theorem, which can be thought of as a far-reaching generalization of the AM--GM inequality.  Next, in Corollary \ref{cor:conv-jack}, we show an analogous statement for Jack polynomials evaluated at any $x \in [0,\infty)^n$.  As a consequence, we obtain Theorem \ref{thm:poly-maj}, which gives Muirhead-type inequalities for both Macdonald and Jack polynomials.  The inequalities for Jack polynomials in particular were separately conjectured by Chen and Sahi \cite{ChenSahi2024} and by McSwiggen and Novak \cite{MN-majorization}. Since Jack polynomials include monomial symmetric polynomials, elementary symmetric polynomials, Schur polynomials, and real and quaternionic zonal polynomials as special cases, this theorem provides a common generalization unifying Muirhead's inequality with Newton's inequalities for elementary symmetric polynomials, together with inequalities for each of these other families of polynomials that were previously proved by Cuttler, Greene and Skandera \cite{CGS}, Sra \cite{Sra1}, and McSwiggen and Novak \cite{MN-majorization}.  As a further application, we obtain Theorem \ref{thm:poly-weak-maj}, which gives inequalities for Jack polynomials that characterize the \emph{weak} majorization order on partitions.  This result was conjectured by Chen and Sahi \cite{ChenSahi2024} and generalizes a theorem on Schur polynomials due to Khare and Tao \cite{KhareTao2021}.  Our final two results concern Heckman--Opdam hypergeometric functions, which play a prominent role in representation theory and integrable systems and are closely related to Jack polynomials.  Theorem \ref{thm:conv-HGF} shows that for all multiplicity parameters $k \ge 0$ and all arguments $x \in \R^n$, the type-$A$ Heckman--Opdam hypergeometric function $F_{k,s}(x)$ is a log-convex and Schur-convex function of the spectral parameter $s \in \R^n$.  Corollary \ref{cor:HGF-maj} then gives Muirhead-type inequalities for $F_{k,s}(x)$, resolving the type-$A$ case of a conjecture by McSwiggen and Novak \cite{MN-majorization}.

\section{Main results}

\subsection{Macdonald and Jack polynomials}

For $q,t \in (0,1)$ and a partition $\lambda \in \mathbb{Y}_n$, let $\Omega_\lambda(x; q,t)$ be the Macdonald polynomial normalized to 1 at the principal specialization $t^\delta = (t^{n-1}, t^{n-2}, \hdots, 1)$, that is,
\[
\Omega_\lambda(x; q,t) = \frac{P_\lambda(x; q,t)}{P_\lambda(t^\delta ; q,t)},
\]
where $P_\lambda$ is the monic normalization of the Macdonald polynomial as defined in \cite[Chapter VI.4]{Macdonald}.

For $a > 0$ and $q,t \in (0,1)$, define the scaled, $t$-shifted, dominant $q$-lattice
        \[
        \mathcal{L}_n = \mathcal{L}^{q,t,a}_n = \big \{ (a q^{-\mu_1}, a q^{-\mu_2} t, \hdots, a q^{-\mu_n} t^{n-1}) \ : \ (\mu_1 \ge \hdots \ge \mu_n) \in \Z^n \big \} \subset \R^n.
        \]

\begin{thm} \label{thm:conv-mac}
    For all $q,t \in (0,1)$, $a > 0$, and $x \in \mathcal{L}^{q,t,a}_n$, the function $\lambda \mapsto \Omega_\lambda(x; q, t)$ is log-convex and Schur-convex.
\end{thm}

As in the case of power sums, the proof below deduces Schur-convexity of the map  $\lambda \mapsto \Omega_\lambda(x; q, t)$ on $\Y_n$ from a combination of log-convexity and symmetry of a suitable extension $s\to f_s(x)$  defined on all of $\R^n$. However, the extension involves a \emph{shift}, so that we have $\Omega_\lambda(x; q,t)=f_{\lambda + \theta \delta}(x)$ where $\theta = \log(t) / \log(q)$. Due to the translation invariance of the majorization order, this turns out not to pose an issue.

\begin{proof}
    We first prove log-convexity using an interlacing integral formula for Macdonald polynomials due to Okounkov and Olshanski.  Okounkov gave a formula for a Macdonald polynomial in $n+1$ variables as an interlacing $q$-integral of a Macdonald polynomial in $n$ variables \cite[Theorem I]{Okounkov2000}, and Olshanski showed that at specializations in $\mathcal{L}_{n+1}$, Okounkov's $q$-integral actually reduces to a finite, positive sum \cite[Theorem A]{OlshanskiMacdonald}. In our notation, this sum can be written
    \begin{equation} \label{eqn:OO-sum}
        \Omega^{(n+1)}_{(\lambda,0)}(y; q,t) = \sum_{x \in \mathcal{L}_n} \Lambda_{n}^{n+1} (y,x) \,  \Omega^{(n)}_\lambda(x; q,t), \qquad y \in \mathcal{L}_{n+1},
    \end{equation}
    where $\Omega^{(n)}_\lambda$ indicates the normalized Macdonald polynomial in $n$ variables, and the coefficients $\Lambda_{n}^{n+1} (y,x)$ are nonnegative and vanish for all but finitely many $x$. After normalizing at the principal specialization, the identity for monic Macdonald polynomials
    \[
    P^{(n+1)}_{(\lambda_1, \hdots, \lambda_{n+1})}(y; q,t) = (y_1 \cdots y_{n+1})^{\lambda_{n+1}} \cdot P^{(n+1)}_{(\lambda_1 - \lambda_{n+1}, \, \hdots, \, \lambda_n - \lambda_{n+1}, \, 0)}(y; q,t)
    \]
    becomes
    \[
    \Omega^{(n+1)}_{(\lambda_1, \hdots, \lambda_{n+1})}(y; q,t) = \left(\frac{y_1 \cdots y_{n+1}}{t^{n(n+1)/2}}\right)^{\lambda_{n+1}} \cdot \Omega^{(n+1)}_{(\lambda_1 - \lambda_{n+1}, \, \hdots, \, \lambda_n - \lambda_{n+1}, \, 0)}(y; q,t),
    \]
    so that we can rewrite (\ref{eqn:OO-sum}) as
    \begin{equation} \label{eqn:OO-sum2}
        \Omega^{(n+1)}_{(\lambda_1, \hdots, \lambda_{n+1})}(y; q,t) = \left(\frac{y_1 \cdots y_{n+1}}{t^{n(n+1)/2}}\right)^{\lambda_{n+1}} \sum_{x \in \mathcal{L}_n} \Lambda_{n}^{n+1} (y,x) \,  \Omega^{(n)}_{(\lambda_1 - \lambda_{n+1}, \, \hdots, \, \lambda_n - \lambda_{n+1})}(x; q, t), \qquad y \in \mathcal{L}_{n+1}.
    \end{equation}
    Iteratively applying the above formula and recalling that $\Omega^{(1)}_\lambda(x;q,t) = x^\lambda$ for $\lambda \in \Z_{\ge 0}$, we find that we can write
    \begin{equation} \label{eqn:Omg-possum}
    \Omega^{(n)}_\lambda(x; q,t) = \sum_{v \in F_x} c_v q^{\langle \lambda,v \rangle}, \qquad x \in \mathcal{L}_n,
    \end{equation}
    where $F_x$ is a finite set of points in $\R^n$ that depends on the specialization $x \in \mathcal{L}_n$, and each $c_\nu$ is a positive coefficient depending on $x$, $q$ and $t$.  The right-hand side above is manifestly log-convex in $\lambda$, which completes the proof of the first claim.

    We next show Schur-convexity.  To this end, we first extend the map $\lambda \mapsto \Omega_\lambda(x; q, t)$ to a function on all of $\R^n$. Let $\theta = \log(t) / \log(q)$ and define, for $s \in \R^n$ and $x \in \mathcal{L}_n$,
    \begin{equation} \label{eqn:f-def}
    f_s(x) = \sum_{v \in F_x} c_v q^{\langle s - \theta \delta,v \rangle}, \qquad x \in \mathcal{L}_n,
    \end{equation}
    with $F_x$ and $c_v$ as in (\ref{eqn:Omg-possum}), so that $f_{\lambda + \theta \delta}(x) = \Omega_\lambda(x; q,t)$ for any partition $\lambda$.
    
    We claim that $f_s(x)$ is a symmetric function of $(s_1, \hdots, s_n)$.  This is a consequence of the fact that Macdonald polynomials can be regarded as \emph{shifted} symmetric functions of the spectral parameter $\lambda$; in particular, Okounkov's binomial formula \cite[equation (1.11)]{Okounkov1997binomial} gives
    \begin{equation} \label{eqn:Ok-binom}
    \Omega_\lambda(x; q, t) = \sum_{\mu \subseteq \lambda} 
\frac{t^{(1-n)|\lambda|} P^*_\mu(q^\lambda; q, t)}{P^*_\mu(q^\mu; q, t) \cdot P_\mu(1, t^{-1}, \ldots, t^{1-n}; q, t)}
\cdot P^*_\mu(x_1, x_2 t, \ldots, x_n t^{n-1} ; q, t),
    \end{equation}
    where $q^\lambda = (q^{\lambda_1}, \hdots, q^{\lambda_n})$, $|\lambda| = \lambda_1 + \cdots + \lambda_n$, $\mu \subseteq \lambda$ indicates containment of partitions (that is, $\mu_i \le \lambda_i$ for $i = 1, \hdots, n$), and $P_\mu^*$ is the shifted Macdonald polynomial defined and studied in \cite{Knop1997, KnopSahi1996, Sahi1996}.  The crucial property of $P_\mu^*$ for our purposes is that $P^*_\mu(q^\lambda; q, t)$ is a symmetric function of the \emph{shifted} variables $q^\lambda t^\delta = (q^{\lambda_1} t^{n-1}, \hdots, q^{\lambda_n})$.  Taking the base-$q$ logarithm of each variable, we can equivalently regard $P^*_\mu(q^\lambda; q, t)$ as a symmetric function of $(\lambda_1 + \theta(n-1), \lambda_2 + \theta(n-2), \hdots, \lambda_n) = \lambda + \theta \delta$. Since $f_{\lambda + \theta \delta}(x) = \Omega_\lambda(x; q,t)$, the right-hand side of (\ref{eqn:Ok-binom}) depends on $\lambda$ only through $t^{(1-n)|\lambda|} P^*_\mu(q^\lambda; q, t)$, and the factor $t^{(1-n)|\lambda|}$ is also a symmetric function of $\lambda + \theta \delta$, we find that $f_s(x)$ is symmetric in $(s_1, \hdots, s_n)$ whenever $s = \lambda + \theta \delta$ for some partition $\lambda$.  Moreover, from the definition (\ref{eqn:f-def}), the function $s \mapsto f_s(x)$ is an exponential polynomial with real exponents; the fact that it is symmetric when evaluated at points $s = \lambda + \theta \delta$ for all partitions $\lambda$ thus implies that it is a globally symmetric function on all of $\R^n$, as desired.

    The function $s \mapsto f_s(x)$ is manifestly log-convex from the definition (\ref{eqn:f-def}), and we have now shown that it is also symmetric.  A symmetric, convex function is Schur-convex \cite[Proposition C.2, p.~97]{MO}; since log-convexity is stronger than convexity, we have proved that the map $s \mapsto f_s(x)$ is Schur-convex for all $x \in \mathcal{L}_n$.  Since the majorization order on $\R^n$ is translation invariant, a partition $\lambda$ majorizes another partition $\mu$ if and only if the vector $\lambda + \theta \delta$ majorizes $\mu + \theta \delta$.  The Schur-convexity of $s \mapsto f_s(x)$ thus implies that the function on partitions $\lambda \mapsto f_{\lambda + \theta \delta}(x) = \Omega_\lambda(x; q,t)$ is Schur-convex, completing the proof of the theorem.
\end{proof}

\begin{remark}
    Muirhead's inequality (\ref{eqn:muirhead}) and many analogous inequalities in the literature, including the results on Jack polynomials in Corollary \ref{cor:conv-jack} and Theorem \ref{thm:poly-maj} below, establish Schur-convexity in the parameter $\lambda$ for some family of polynomials at all nonnegative specializations $x \in [0,\infty)^n$.  It is natural to wonder whether such a result might also hold for Macdonald polynomials---that is, whether Theorem \ref{thm:conv-mac}, which deals only with specializations $x \in \mathcal{L}_n$, could be upgraded to a statement for all $x \in [0,\infty)^n$.  In fact this cannot be done.  As shown in \cite[Example 5.1]{ChenKhareSahi}, there are positive specializations $x$ for which the map $\lambda \mapsto \Omega_\lambda(x;q,t)$ fails to be Schur-convex (and thus is also not log-convex, since log-convexity together with the shifted symmetry discussed above would imply Schur-convexity). Nevertheless, as we show below via a density argument in an appropriate limit, the statement for Macdonald polynomials at $x \in \mathcal{L}_n$ yields a corresponding statement for Jack polynomials at all nonnegative specializations.
\end{remark}

\begin{remark}
    Although Theorem \ref{thm:conv-mac} above and Theorem \ref{thm:poly-maj} below only consider $q,t \in (0,1)$, identical statements hold when both $q$ and $t$ belong to $(1, \infty)$, by the parameter inversion identity for Macdonald polynomials \cite[Chapter VI, (4.14)]{Macdonald}:
    \[
    P_\lambda(x; q^{-1},t^{-1}) = P_\lambda(x; q,t).
    \]
    Normalizing at the principal specialization gives
    \begin{align} \begin{split}
    \Omega_\lambda(x; q^{-1},t^{-1}) &= \frac{P_\lambda(x; q^{-1},t^{-1})}{P_\lambda((t^{-1})^\delta ; q^{-1},t^{-1})} = \frac{P_\lambda(x; q,t)}{P_\lambda((t^{-1})^\delta ; q,t)} = t^{(n-1)|\lambda|} \frac{P_\lambda(x; q^{-1},t^{-1})}{P_\lambda(t^\delta ; q^{-1},t^{-1})} \\ &= t^{(n-1)|\lambda|} \Omega_\lambda(x; q,t).
    \end{split} \end{align}
    Since $|\lambda| = |\mu|$ for $\lambda$ majorizing $\mu$, the inversion $(q,t) \mapsto (q^{-1},t^{-1})$ therefore preserves log-convexity and Schur-convexity in the parameter at any fixed specialization.  Moreover, although the theorem statements only discuss specializations $x \in \mathcal{L}_n$, and the set $\mathcal{L}_n = \mathcal{L}^{q,t,a}_n$ depends on $q$ and $t$, by symmetry it is clear that log-convexity and Schur-convexity in $\lambda$ in fact hold at all specializations in the $S_n$-orbit of $\mathcal{L}_n$.  The map $(q,t) \mapsto (q^{-1},t^{-1})$ preserves the $S_n$-orbit of $\mathcal{L}_n$ up to a scaling that can be removed by varying $a$, so in either case $q,t \in (0,1)$ or $q,t \in (1, \infty)$ we obtain log-convexity and Schur-convexity at the same set of specializations. In the limits $t \to 1$ (with $q$ fixed) and $q \to 1$ (with $t$ fixed), the Macdonald polynomials degenerate to monomial and elementary symmetric functions respectively; these cases are covered by the results on Jack polynomials below.
\end{remark}

For $0 < \theta < \infty$, the normalized Jack polynomial $\Omega_\lambda(x; \theta)$ is obtained from $\Omega_\lambda(x; q,t)$ via the limit transition
\begin{equation} \label{eqn:mac2jack}
    \Omega_\lambda(x; \theta) = \lim_{q \to 1} \Omega_\lambda(x; q, q^\theta)
\end{equation}
and satisfies $\Omega_\lambda(\bm{1}; \theta) = 1$.  For details on the Macdonald-to-Jack limit see \cite[Chapter VI.10]{Macdonald}; note that the parameter $\theta$ corresponds to $1/\alpha$ in Macdonald's notation. The limits $\theta \to 0$ and $\theta \to \infty$ recover the normalized monomial and elementary symmetric polynomials respectively:
\begin{align} \label{eqn:jackmono}
\Omega_\lambda(x;0) &= \lim_{\theta \to 0} \Omega_\lambda(x; \theta) = \frac{m_\lambda(x)}{m_\lambda(\mathbf{1})}, \\
\label{eqn:jackelem}
\Omega_\lambda(x;\infty) &= \lim_{\theta \to \infty} \Omega_\lambda(x; \theta) = \frac{e_{\lambda'}(x)}{e_{\lambda'}(\mathbf{1})},
\end{align}
where $m_\lambda(x)$ is the monomial symmetric polynomial corresponding to the partition $\lambda$, and $e_{\lambda'}(x)$ is the elementary symmetric polynomial corresponding to the conjugate partition $\lambda'$ with entries $\lambda_j' = \# \{i \, : \, \lambda_i \ge j \}$, $j = 1, \hdots, n$.

\begin{cor} \label{cor:conv-jack}
    For all $\theta \in [0,\infty]$ and $x \in [0,\infty)^n$, the function $\lambda \mapsto \Omega_\lambda(x; \theta)$ is log-convex and Schur-convex.
\end{cor}

\begin{proof}
    We initially suppose that $0 < \theta < \infty$; we will remove this assumption at the end of the proof.
    
    It suffices to show the claim for $x \in \R^n$ with $x_1 > \hdots > x_n > 0$.  The claim for all $x \in [0,\infty)^n$ then follows from the symmetry and continuity of the Jack polynomials. For any such $x$, choose a positive $a < x_n$ so that $\log(x_1 / a) > \hdots > \log(x_n / a) > 0$, and define a sequence of partitions $\mu(k) = (\mu_1(k) \ge \hdots \ge \mu_n(k))$, $k = 1,2,...$, by
    \[
    \mu_j(k) = \left \lfloor k \log\frac{x_j}{a} \right \rfloor, \qquad j = 1, \hdots, n.
    \]
    Take $q = q(k) = e^{-1/k}$, $t = t(k) = q^\theta = e^{-\theta/k}$. Then $q \to 1$ from below as $k \to \infty$, while the points
    \begin{equation*}
    x^{(k)} = \big( x^{(k)}_1, \dots, x^{(k)}_n \big) \in \mathcal{L}_n
    \end{equation*}
    with coordinates
    \[
    x^{(k)}_j = aq^{-\mu_j(k)} t^{j-1} = a \exp \left[ \frac{1}{k} \left( \left \lfloor k \log\frac{x_j}{a} \right \rfloor - (j-1)\theta \right) \right]
    \]
    converge to $x$.  Since (\ref{eqn:mac2jack}) is a limit of polynomials of fixed degree, the coefficients are uniformly bounded as $q \to 1$, and thus the convergence is uniform on compact subsets of $\R^n$.  Accordingly, for any $\lambda \in \mathbb{Y}_n$,
    \begin{equation} \label{eqn:jack-xlim}
    \Omega_\lambda(x; \theta) = \lim_{k \to \infty} \Omega_\lambda(x^{(k)}; e^{-1/k}, e^{-\theta/k}) > 0.
    \end{equation}
    For each $k$, the function $\lambda \mapsto \Omega_\lambda(x^{(k)}; e^{-1/k}, e^{-\theta/k})$ is log-convex and Schur-convex on $\mathbb{Y}_n$ by Theorem \ref{thm:conv-mac}, and (\ref{eqn:Omg-possum}) gives a log-convex extension of this function to $\R^n$.  Since the coefficients in (\ref{eqn:Omg-possum}) are uniformly bounded as $k \to \infty$, in the limit we obtain an extension to $\R^n$ of the map $\lambda \mapsto \Omega_\lambda(x; \theta)$.  A limit of Schur-convex functions on $\mathbb{Y}_n$ is Schur-convex and a limit of log-convex functions on $\R^n$ is log-convex, so we have proved the claim under the assumption that $0 < \theta < \infty$.

    The cases $\theta = 0$ and $\theta = \infty$ now follow immediately because (\ref{eqn:jackmono}) and (\ref{eqn:jackelem}) respectively express $\lambda \mapsto \Omega_\lambda(x;0)$ and $\lambda \mapsto \Omega_\lambda(x;\infty)$ as limits of functions that are both log-convex (on $\R^n$) and Schur-convex (on $\mathbb{Y}_n$).
\end{proof}

\begin{thm} \label{thm:poly-maj}
    For partitions $\lambda, \mu \in \mathbb{Y}_n$, the following are equivalent:
    \begin{enumerate}
        \item $\lambda$ majorizes $\mu$.
        \item For all $q, t \in (0,1)$ and $a > 0$, $\Omega_\lambda(x; q,t) \ge \Omega_\mu(x; q,t)$ for all $x \in \mathcal{L}^{q,t,a}_n$.
        \item There exist $q, t \in (0,1)$ and $a > 0$ such that $\Omega_\lambda(x; q,t) \ge \Omega_\mu(x; q,t)$ for all $x \in \mathcal{L}^{q,t,a}_n$.
        \item For all $\theta \in [0,\infty]$, $\Omega_\lambda(x; \theta) \ge \Omega_\mu(x; \theta)$ for all $x \in [0,\infty)^n$.
        \item There exists $\theta \in [0,\infty]$ such that $\Omega_\lambda(x; \theta) \ge \Omega_\mu(x; \theta)$ for all $x \in [0,\infty)^n$.
    \end{enumerate}
\end{thm}

The possible equivalence of (1), (4) and (5) above was discussed by Sra \cite{Sra1} and was later conjectured to hold by McSwiggen and Novak \cite{MN-majorization} and by Chen and Sahi \cite{ChenSahi2024}.  Theorem \ref{thm:poly-maj} resolves this question.  In the case $\theta = 0$ the inequalities for Jack polynomials recover Muirhead's inequality (\ref{eqn:muirhead}), while for $\theta = \infty$ they recover an analogous inequality for elementary symmetric polynomials shown in \cite{CGS}, which in turn generalizes Newton's inequalities.  (Note that $\lambda$ majorizes $\mu$ if and only if the conjugate partition $\mu'$ majorizes $\lambda'$, so in the inequality for elementary symmetric polynomials the relation to the majorization order is reversed.)  For $\theta = 1$ we obtain an inequality for Schur polynomials that was conjectured in \cite{CGS} and shown in \cite{Sra1}, while for $\theta = 1/2$ or $2$ we obtain inequalities for zonal polynomials shown in \cite{MN-majorization}.

\begin{proof}
    The implications (1) $\implies$ (2) and (1) $\implies$ (4) are the Schur-convexity statements of Theorem \ref{thm:conv-mac} and Corollary \ref{cor:conv-jack} respectively. The implications (2) $\implies$ (3) and (4) $\implies$ (5) are trivial.  It remains to show that (3) $\implies$ (1) and (5) $\implies$ (1).  These implications both follow via a standard asymptotic technique.  We first show (5) $\implies$ (1) and then discuss the modifications required to prove (3) $\implies$ (1).

    For $t \in \R$, take $x = x(t)$ with $x_1 = \cdots = x_r = t$ and $x_{r+1} = \cdots = x_n = 1$ for some $1 \le r \le n$.  At this specialization, $\Omega_\lambda(x(t);\theta)$ is a polynomial of degree $\lambda_1 + \cdots + \lambda_r$ in $t$ with nonnegative coefficients.  To see this, recall that the monic normalization of the Jack polynomial is given by
    \begin{equation} \label{eqn:jack-mono}
    P_\lambda(x;\theta) = m_\lambda(x) + \sum_{\nu \preceq \lambda} c_{\lambda\nu}(\theta) \, m_\nu(x),
    \end{equation}
    where the sum is over partitions $\nu$ majorized by $\lambda$, and $c_{\lambda\nu}(\theta) \ge 0$. At the specialization $x(t)$, the monomial symmetric polynomial $m_\nu(x(t))$ is a polynomial in $t$ of degree $\nu_1 + \cdots + \nu_r$ with nonnegative coefficients. Since $\nu_1 + \cdots + \nu_r \le \lambda_1 + \cdots + \lambda_r$ for each $\nu$ appearing in the sum, the degree of $P_\lambda(x(t);\theta)$ in $t$ is exactly $\lambda_1 + \cdots + \lambda_r$, with all coefficients nonnegative. Since $\Omega_\lambda(x;\theta) = P_\lambda(x;\theta) / P_\lambda(\mathbf{1};\theta)$ and $P_\lambda(\mathbf{1};\theta) > 0$, the same holds for $\Omega_\lambda(x(t);\theta)$.
    
    Now suppose that (1) does not hold, that is, $\lambda$ does not majorize $\mu$.  Then for some $1 \le r \le n$, we have $\mu_1 + \cdots + \mu_r > \lambda_1 + \cdots + \lambda_r$.  As we have just shown, for all $\theta \in [0,\infty]$, $\Omega_\lambda(x(t); \theta)$ and $\Omega_\mu(x(t); \theta)$ are then polynomials in $t$ with nonnegative coefficients and respective, unequal degrees $\lambda_1 + \cdots + \lambda_r$ and $\mu_1 + \cdots + \mu_r$.  Therefore $\Omega_\mu(x(t); \theta) > \Omega_\lambda(x(t); \theta)$ for $t$ sufficiently large, so that (5) cannot hold. This completes the proof that (5) $\implies$ (1).

    The proof that (3) $\implies$ (1) is nearly identical, with one additional subtlety.  We again suppose that $\lambda$ does not majorize $\mu$ and then prove that some of the inequalities asserted in (3) must fail.  The monic Macdonald polynomials can be expanded as
    \begin{equation} \label{eqn:mac-mono}
    P_\lambda(x;q,t) = m_\lambda(x) + \sum_{\nu \preceq \lambda} c_{\lambda\nu}(q,t) \, m_\nu(x)
    \end{equation}
    with $c_{\lambda\nu}(q,t) \ge 0$, analogously to (\ref{eqn:jack-mono}), so there is only one obstacle to applying the same argument: the inequalities in (3) are only required to hold at specializations in $\mathcal{L}_n$ rather than the entire positive orthant. Therefore, instead of merely considering asymptotics along a ray, we need to find a sequence of specializations $x(k) \in \mathcal{L}_n$, $k = 1, 2, \hdots$, such that $\Omega_\mu(x(k); q,t) > \Omega_\lambda(x(k); q,t)$ for $k$ sufficiently large.
    
    For this we can take $x(k) = a \cdot (q^{-\nu_1(k)}, q^{-\nu_2(k)}t, \dots, q^{-\nu_n(k)} t^{n-1})$ with $\nu_1(k) = \cdots = \nu_r(k) = k$, $\nu_{r+1}(k) = \cdots = \nu_n(k) = 0$, for $1 \le r \le n$ such that $\mu_1 + \cdots + \mu_r > \lambda_1 + \cdots + \lambda_r$.  That is,
    \[
    x(k) = a \cdot (q^{-k}, q^{-k} t, \hdots, q^{-k}t^{r-1}, t^r, t^{r+1}, \hdots, t^{n-1}).
    \]
    To see that $\Omega_\mu(x(k); q,t) > \Omega_\lambda(x(k); q,t)$ for $k$ sufficiently large, consider $y = y(s) \in \R^n$ with $y_1 = \cdots = y_r = s$ and $y_{r+1} = t^r, \ y_{r+2} = t^{r+1}, \ \dots,\ y_n = t^{n-1}$. Since the monomial symmetric polynomials are increasing functions of each argument, the same holds for the Macdonald polynomials (in any normalization) by the expansion (\ref{eqn:mac-mono}). Because $t \in (0,1)$, the first $r$ coordinates of $x(k)$, namely $aq^{-k}t^{j-1}$ for $j = 1, \hdots, r$, satisfy $aq^{-k} t^{r-1} \le aq^{-k}t^{j-1} \le aq^{-k}$, while the remaining $n-r$ coordinates coincide with those of $ay(s)$.  Therefore
    \begin{align*}
        \Omega_\lambda(a y(q^{-k}t^{r-1}); q, t) \le \Omega_\lambda(x(k); q, t) \le \Omega_\lambda(a y(q^{-k}); q, t),
    \end{align*}
    with analogous inequalities for $\Omega_\mu$, and thus it suffices to show that
    \[
    \Omega_\mu(a y(st^{r-1}); q,t) > \Omega_\lambda(a y(s); q,t)
    \]
    for $s$ sufficiently large. This must be the case because the left- and right-hand sides above are polynomials in $s$ with positive coefficients and respective degrees $\mu_1 + \cdots + \mu_r > \lambda_1 + \cdots + \lambda_r$, completing the proof of the final remaining implication.
\end{proof}

The above method for proving the ``easy direction'' of Muirhead-type theorems, via an asymptotic construction of points where the inequalities fail if $\lambda$ does not majorize $\mu$, has been widely used in the literature in individual cases.  However, to our knowledge, a general version of the technique has not been described. In hopes that it will be useful for future work in this area, we record a set of sufficient conditions on the polynomials and specializations under consideration that ensure the method is applicable.

\begin{prop}
    Let $\lambda, \mu \in \mathbb{Y}_n$ and suppose that $\lambda$ does not majorize $\mu$, so that $\mu_1 + \cdots + \mu_r > \lambda_1 + \cdots + \lambda_r$ for some $1 \le r \le n$. Let $f, g \in \R[x_1, \hdots, x_n]$ be two polynomials with expansions
    \begin{align}
    \label{eqn:f-mono}
        f(x) &= \sum_{\nu \preceq \lambda} c_{\lambda\nu} m_\nu(x), \\
        \label{eqn:g-mono}
        g(x) &= \sum_{\nu \preceq \mu} c_{\mu\nu} m_\nu(x),
    \end{align}
    where the sums are over partitions $\nu$ majorized by $\lambda$ and $\mu$ respectively, all of the coefficients $c_{\lambda \nu}$, $c_{\mu \nu}$ are nonnegative, and the leading coefficients $c_{\lambda \lambda}$, $c_{\mu\mu}$ are strictly positive.  For $k =1,2,\hdots$, take $x(k) \in \R^n$ with $x_1(k) \ge \cdots \ge x_n(k) > \varepsilon > 0$ such that
    \[
        1/C < \frac{x_1(k)}{x_r(k)} < C, \quad x_{r+1}(k) < C \qquad \text{for some } C>0,
    \]
    and $x_r(k)/x_{r+1}(k) \to \infty$.
    Then $g(x(k)) > f(x(k))$ for $k$ sufficiently large.
\end{prop}

\begin{proof}
    As argued in the proof of Theorem \ref{thm:poly-maj}, the monomial expansions (\ref{eqn:f-mono}) and (\ref{eqn:g-mono}) imply the following properties of $f$ and $g$: they are homogeneous of degrees $|\lambda| = \lambda_1 + \cdots + \lambda_n$ and $|\mu| = \mu_1 + \cdots + \mu_n$ respectively; they are increasing functions of each $x_i$; and, for $y = y(s) \in \R^n$ with $y_1 = \cdots = y_r = s$ and $y_{r+1} = \cdots = y_n = 1$, $f(y(s))$ and $g(y(s))$ are polynomials in $s$ with nonnegative coefficients and respective degrees $\lambda_1 + \cdots + \lambda_r$ and $\mu_1 + \cdots + \mu_r$.

    The hypotheses imply that $x_i(k) \le C x_r(k)$ for $1 \le i \le r$ and $x_j(k) < C$ for $r+1 \le j \le n$.  By monotonicity in the individual coordinates and homogeneity, we then have
    \[
        f(x(k)) \;\le\; f\bigl(C x_r(k), \hdots, C x_r(k), C, \hdots, C\bigr)
        \;=\; C^{|\lambda|} f\bigl(y(x_r(k))\bigr).
    \]
    Similarly, $x_i(k) \ge x_r(k)/C$ for $1 \le i \le r$ and $x_j(k) > \varepsilon$ for $r+1 \le j \le n$, so
    \[
        g(x(k)) \;\ge\; g\bigl(x_r(k)/C, \ldots, x_r(k)/C, \varepsilon, \ldots,
        \varepsilon\bigr)
        \;=\; \varepsilon^{|\mu|} g\bigg(y \Big(\frac{x_r(k)}{C \varepsilon}\Big)\bigg).
    \]
    By assumption $x_r(k) / x_{r+1}(k) \to \infty$ and $x_{r+1}(k) > \varepsilon$, so $x_r(k) \to \infty$.  From the bounds above, it therefore suffices to show that
    \[
    \varepsilon^{|\mu|} g\bigg(y \Big(\frac{s}{C \varepsilon}\Big)\bigg) \;>\; C^{|\lambda|} f\bigl(y(s)\bigr)
    \]
    for $s$ sufficiently large, which follows because the left- and right-hand sides above are polynomials in $s$ with positive coefficients and respective degrees $\mu_1 + \cdots + \mu_r > \lambda_1 + \cdots + \lambda_r$.
\end{proof}

A partition $\lambda$ is said to \emph{weakly majorize} a partition $\mu$ if 
\begin{equation} \label{eqn:weakmaj-def}
\sum_{i=1}^r \lambda_i \ge \sum_{i=1}^r \mu_i \qquad \text{for all } 1 \le r \le n.
\end{equation}
In other words, $\lambda$ majorizes $\mu$ if and only if $\lambda$ weakly majorizes $\mu$ and additionally the sums of the entries of both partitions are equal.  Theorem \ref{thm:poly-maj} implies the following further result, which was conjectured by Chen and Sahi \cite[Conjecture 1(2)]{ChenSahi2024} and generalizes an earlier result on Schur polynomials by Khare and Tao \cite{KhareTao2021}.

\begin{thm} \label{thm:poly-weak-maj}
    For partitions $\lambda, \mu \in \mathbb{Y}_n$, the following are equivalent:
    \begin{enumerate}
        \item $\lambda$ weakly majorizes $\mu$.
        \item For all $\theta \in [0,\infty]$, $\Omega_\lambda(x; \theta) \ge \Omega_\mu(x; \theta)$ for all $x \in [1,\infty)^n$.
        \item There exists $\theta \in [0,\infty]$ such that $\Omega_\lambda(x; \theta) \ge \Omega_\mu(x; \theta)$ for all $x \in [1,\infty)^n$.
    \end{enumerate}
\end{thm}

\begin{proof}
    The implication (2) $\implies$ (3) is trivial, and the implication (3) $\implies$ (1) follows from the same asymptotic analysis for Jack polynomials detailed in the proof of Theorem \ref{thm:poly-maj}.  It remains to show that (1) $\implies$ (2).
    
    Suppose that $\lambda$ weakly majorizes $\mu$.  By \cite[Lemma 6.6]{ChenSahi2024}, this means that $\lambda$ contains a partition $\nu$ (that is, $\lambda_i \ge \nu_i$ for all $i = 1, \hdots, n$) such that $\nu$ majorizes $\mu$. In \cite[Theorem 6.1]{ChenSahi2024}, Chen and Sahi gave a characterization of the containment order on partitions in terms of Jack polynomial differences: in particular, when $\lambda$ contains $\nu$, the difference $\Omega_\lambda(y+\mathbf{1}; \theta) - \Omega_\nu(y+\mathbf{1}; \theta)$ is Jack positive, meaning that it has an expansion in Jack polynomials $\Omega_\nu(y;\theta)$ where the coefficient of each $\Omega_\nu$ is a rational function $f_\nu(\theta)/g_\nu(\theta)$, and the polynomials $f_\nu$ and $g_\nu$ have nonnegative integer coefficients.  This implies that $\Omega_\lambda(y+\mathbf{1}; \theta) \ge \Omega_\nu(y+\mathbf{1}; \theta)$ for $y \in [0,\infty)^n$.
    
    Combining this inequality with the majorization inequality for Jack polynomials in Theorem \ref{thm:poly-maj}, and taking $x \in [1,\infty)^n$ so that $x = y+\mathbf{1}$ for some $y \in [0,\infty)^n$, we then have
    \[
    \Omega_\lambda(x; \theta) \ge \Omega_\nu(x; \theta) \ge \Omega_\mu(x;\theta)
    \]
    as desired.
\end{proof}

\subsection{Heckman--Opdam hypergeometric functions}

The Heckman--Opdam hypergeometric functions are multivariable generalizations of the classical Gauss hypergeometric function that are associated with root systems.  They were originally constructed and studied in the series of papers \cite{RSHF1, RSHF2, RSHF3, RSHF4}; see \cite{AnkerDunklNotes, HS} for reviews.  Here we show a log-convexity property and majorization inequalities for the Heckman--Opdam hypergeometric functions of type $A$, which we now define.

For $\xi \in \R^n$ and a real parameter $k \ge 0$, the {\it Cherednik operator} $T_{k,\xi}$ is the differential-difference operator defined by
\begin{equation} \label{eqn:cherednik-op-def}
T_{k,\xi} f(x) = \partial_\xi f(x) + k \sum_{1 \le i < j \le n} 
\frac{\xi_i - \xi_j}{1-e^{-(x_i - x_j)}} 
\big[ f(x) - f(s_{ij} x) \big] - k \langle \rho, \xi \rangle f(x), \qquad f \in C^1(\R^n),
\end{equation}
where $s_{ij}$ denotes the transposition of coordinates $x_i$ and $x_j$, and
\[
\rho = \frac{1}{2} \sum_{1 \le i < j \le n} (e_i - e_j) = \frac{1}{2} \big( n-1, n-3, n-5, \hdots, -(n-3), -(n-5) \big),
\]
where $e_1, \hdots, e_n$ are the standard basis vectors in $\R^n$.

For each $s \in \R^n$ (in fact for $s \in \C^n$, though here we consider real $s$ only), there is a unique function $G_{k,s} \in C^\infty(\R^n)$ satisfying the system of differential-difference 
equations
\begin{equation} \label{eqn:hyperbolic-eigproblem}
T_{k, \xi} G_{k, s} = \langle s, \xi \rangle G_{k, s} 
\qquad \text{ for all } \, \xi \in \R^n,
\end{equation}
and normalized such that $G_{k,s}(0) = 1$. The {\it Heckman--Opdam hypergeometric function} $F_{k, s}$ 
is the symmetrization of $G_{k,s}$ over $S_n$:
\begin{equation} \label{eqn:HGF-def}
F_{k, s}(x) = \frac{1}{n!} \sum_{\sigma \in S_n} G_{k, s}(\sigma(x)).
\end{equation}
It is $S_n$-invariant in both $x$ and $s$. When $s = \lambda + k \rho$ with $\lambda \in \mathbb{Y}_n$, a logarithmic change of variables recovers the normalized Jack polynomial with $\theta = k$.  Setting $y_i = e^{x_i}$, we have:
\begin{equation}
    F_{k,\lambda + k \rho}(x) = \Omega_\lambda(y; k).
\end{equation}

\begin{thm} \label{thm:conv-HGF}
    For all $k \ge 0$ and $x \in \R^n$, the function $s \mapsto F_{k,s}(x)$, $s \in \R^n$, is log-convex and Schur-convex.
\end{thm}

Theorem \ref{thm:conv-HGF} can be deduced from Theorem \ref{thm:conv-mac} via a limit transition that recovers the type-$A$ Heckman--Opdam hypergeometric function as a degeneration of Macdonald polynomials; see for example the limit formula given in \cite[Proposition B.2]{BorodinGorin2015}.  However, it is simpler to give a direct proof using an interlacing integral formula for the hypergeometric function due to Amri and Bedhiafi \cite{AmriBedhiafi}, which can be regarded as a corresponding degeneration of the Okounkov--Olshanski formula (\ref{eqn:OO-sum}).

\begin{proof}
We first prove the theorem for $k > 0$ and $x \in \R^n$ with $x_1 > x_2 > \cdots > x_n$ and $x_1 + \cdots + x_n = 0$.  Later we will remove these assumptions.  Define
\begin{equation*}
V_n(x) = \prod_{1 \le i < j \le n} (e^{x_i} - e^{x_j}), \qquad \quad W_k(x, \nu) = \prod_{\substack{1 \le i \le n \\ 1 \le j \le n-1}} |e^{x_i} - e^{\nu_j}|^{k-1},
\end{equation*}
where $\nu = (\nu_1, \ldots, \nu_{n-1}) \in \R^{n-1}$ satisfies the interlacing conditions
$x_{j+1} \le \nu_j \le x_j$ for $j = 1, \ldots, n-1$.
With the stated assumptions on $x$ and $k$, \cite[Theorem 2.1]{AmriBedhiafi} gives the recursive formula
\begin{equation} \label{eqn:amri-bedhiafi}
F^{(n)}_{k,s}(x) = \frac{\Gamma(nk)}{\Gamma(k)^n V_n(x)^{2k-1}}
\int_{x_2}^{x_1} \cdots \int_{x_n}^{x_{n-1}}
F^{(n-1)}_{k,\, \pi_{n-1}(s)}\!\big(\pi_{n-1}(\nu)\big)
e^{|\nu|\left(1 - \frac{nk}{2} + \frac{|s|}{n-1}\right)}
V_{n-1}(\nu)\, W_k(x, \nu)\, d\nu,
\end{equation}
where $\pi_{n-1}(s) = s - \frac{1}{n}\bigl(\sum_{j=1}^n s_j\bigr)\bm{1}$
denotes the orthogonal projection onto $\{x \in \R^n : \sum_i x_i = 0\}$,
and $|\nu| = \nu_1 + \cdots + \nu_{n-1}$, $|s| = s_1 + \cdots + s_n$.

We now prove log-convexity of $s \mapsto F^{(n)}_{k,s}(x)$ by induction on $n$. For the base case $n = 1$, the only argument $x \in \R^n$ whose coordinates sum to $0$ is $x=0$.  At this point we have $F^{(1)}_{k,s}(x) \equiv 1$ for all $s$, which is trivially log-convex.

Now assume that $s \mapsto F^{(n-1)}_{k,s}(y)$ is log-convex for all $y \in \R^{n-1}$.
Fix $x \in \R^n$ with $x_1 > \cdots > x_n$ and $x_1 + \cdots + x_n = 0$, and fix $\nu$ in the interlacing region
$x_{j+1} \le \nu_j \le x_j$.  The
inductive hypothesis implies that the map
\[
s \mapsto F^{(n-1)}_{k,\,\pi_{n-1}(s)}\!\big(\pi_{n-1}(\nu)\big)
\]
is log-convex in $s$: it is a composition of the log-convex function
$t \mapsto F^{(n-1)}_{k,t}(\pi_{n-1}(\nu))$ with the linear map $s \mapsto \pi_{n-1}(s)$. Moreover, the factor
\[
s \mapsto e^{|\nu|\left(1 - \frac{nk}{2} + \frac{|s|}{n-1}\right)}
\]
is log-convex in $s$, since $|s| = s_1 + \cdots + s_n$ is linear in $s$.
Since the product of log-convex functions is log-convex, the integrand
\[
s \mapsto F^{(n-1)}_{k,\,\pi_{n-1}(s)}\!\big(\pi_{n-1}(\nu)\big)\,
e^{|\nu|\left(1 - \frac{nk}{2} + \frac{|s|}{n-1}\right)}
V_{n-1}(\nu)\, W_k(x,\nu)
\]
is log-convex and positive in $s$ for each fixed $\nu$ in the interlacing region.
Since the prefactor $\Gamma(nk)/(\Gamma(k)^n V_n(x)^{2k-1})$ is positive and
independent of $s$, it follows from (\ref{eqn:amri-bedhiafi}) that $s \mapsto F^{(n)}_{k,s}(x)$
is a positive linear combination (integral) of log-convex functions and is therefore log-convex.

We now remove the assumption that $x_1 + \cdots + x_n = 0$.  For $\xi = e_1 + \cdots + e_n$, the Cherednik operator $T_{k,\xi}$ is just $\partial_1 + \cdots + \partial_n$, where $\partial_i$ is the partial derivative in the $i$th coordinate direction.  From (\ref{eqn:hyperbolic-eigproblem}) and (\ref{eqn:HGF-def}), we then have
\[
( \partial_1 + \cdots + \partial_n ) F_{k,s}(x) = (s_1 + \cdots + s_n) F_{k,s}(x),
\]
which implies
\[
F_{k,s}(x) = e^{(s_1 + \cdots + s_n)(x_1 + \cdots + x_n)} F_{k,s}(\pi_{n-1}(x)).
\]
The right-hand side above is a product of two log-convex functions in $s$ and is therefore log-convex.

This proves log-convexity under the assumptions that $k > 0$ and $x_1 > \cdots > x_n$. The full claim for all $k \ge 0$ and $x \in \R^n$ then follows by the symmetry of $F^{(n)}_{k,s}(x)$
in $x$, continuity of $F^{(n)}_{k,s}(x)$
in $x$ and $k$, and the fact that log-convexity is preserved under pointwise limits.

Finally, Schur-convexity then follows as in the proof of Theorem \ref{thm:conv-mac}.  The function $s \mapsto F^{(n)}_{k,s}(x)$ is symmetric in $s$ and we have just proved that it is also convex. By \cite[Proposition C.2, p.~97]{MO} it is therefore Schur-convex.
\end{proof}

The following corollary establishes the type-$A$ case of a conjecture by McSwiggen and Novak \cite[Conjecture 4.7]{MN-majorization}.

\begin{cor} \label{cor:HGF-maj}
    For any $r, s \in \R^n$, the following are equivalent:
    \begin{enumerate}
        \item $r$ majorizes $s$.
        \item For all $k \ge 0$, $F_{k, r}(x) \ge F_{k,s}(x)$ for all $x \in \R^n$.
        \item There exists $k \ge 0$ such that $F_{k, r}(x) \ge F_{k,s}(x)$ for all $x \in \R^n$.
    \end{enumerate}
\end{cor}

\begin{proof}
    The implication (1) $\implies$ (2) is the Schur-convexity statement of Theorem \ref{thm:conv-HGF}.  The implication (2) $\implies$ (3) is trivial.  The implication (3) $\implies$ (1) is the type-$A$ case of \cite[Proposition 4.8]{MN-majorization}.
\end{proof}

\section*{Acknowledgements}

The work of C.M. is partially supported by the National Science and Technology Council of Taiwan under grant number 113WIA0110762.  The work of S.S. is partially supported by the Simons Foundation under grant number 00006698.  The authors would like to thank Suvrit Sra for identifying a misstatement in an earlier version of this manuscript, documented in \cite{sra2026gpt}, stemming from an incorrect ordering of the $t$-exponents in the definition of the lattice points $\mathcal{L}_n$.  The authors would also like to thank Swee Hong Chan for bringing this matter to their attention.

\bibliography{refs}
\bibliographystyle{amsplain}

\end{document}